\documentclass[11pt,a4paper]{article}

\usepackage[margin=23mm]{geometry}
\usepackage{amsmath,amssymb,amsthm,mathtools}
\usepackage{enumitem}
\usepackage{hyperref}
\hypersetup{
colorlinks=true,
linkcolor=blue,
citecolor=blue,
urlcolor=blue
}
\usepackage{microtype}
\usepackage{authblk}
\usepackage{aliascnt}
\usepackage[noabbrev]{cleveref}
\allowdisplaybreaks

\newtheorem{theorem}{Theorem}[section]
\newtheorem{lemma}[theorem]{Lemma}
\newtheorem{proposition}[theorem]{Proposition}
\newtheorem{corollary}[theorem]{Corollary}
\newtheorem{problem}[theorem]{Problem}
\theoremstyle{remark}
\newtheorem{remark}[theorem]{Remark}

\newcommand{\R}{\mathbb R}
\newcommand{\cA}{\mathcal A}
\newcommand{\cD}{\mathcal D}
\newcommand{\cF}{\mathcal F}
\newcommand{\cG}{\mathcal G}
\newcommand{\cH}{\mathcal H}
\newcommand{\cL}{\mathcal L}
\newcommand{\Span}{\operatorname{span}}

\title{Extremal Families for Matchings in Permutations}
\author[1]{Mengyu Cao\thanks{E-mail: \texttt{myucao@ruc.edu.cn}. Supported by the National Natural Science Foundation of China (12301431) and Beijing Natural Science Foundation (1262010).}}
\author[2]{Haixiang Zhang\thanks{Corresponding author. E-mail: \texttt{zhang-hx22@mails.tsinghua.edu.cn}.}}

\affil[1]{\small Institute for Mathematical Sciences, Renmin University of China, Beijing 100086, China}
\affil[2]{\small Department of Mathematical Sciences, Tsinghua University, Beijing 100084, China}
\date{}

\begin{document}
\maketitle
\begin{abstract}
Two permutations $\sigma,\tau\in S_n$ are called disjoint if the composition
$\sigma\tau^{-1}$ has no fixed point.  If a family $\cF\subseteq S_n$
contains no $s$ pairwise disjoint permutations, then a simple averaging
argument gives $|\cF|\leq(s-1)(n-1)!$.  Inozemtsev, Kolupaev and Kupavskii
characterized the equality cases in the range
$
s\leq n/(2^{17}\log n).
$  We characterize all equality cases throughout the range $2\le s\le n$: equality holds
if and only if $\cF$ is a union of $(s-1)$ pairwise disjoint
$1$-cosets.  We also
prove the linear statement underlying this classification: a real-valued
function on $S_n$ has constant sum on every one-factorization if and only if it lies in the span of the indicators of the $1$-cosets.
The proof is combinatorial and applies to every
order, with a few small orders handled separately.
\end{abstract}


\section{Introduction}

For a family $\cA\subseteq\binom{[N]}{k}$, let $\nu(\cA)$ be the maximum
number of pairwise disjoint members of $\cA$.  The Erd\H{o}s Matching
Conjecture~\cite{ErdosMatching} asserts that, whenever $N\geq sk$ and
$\nu(\cA)<s$,
\[
   |\cA|
   \leq
   \max\left\{
      \binom{N}{k}-\binom{N-s+1}{k},
      \binom{sk-1}{k}
   \right\}.
\]
The two terms are attained, respectively, by all $k$-sets meeting a fixed
$(s-1)$-set and by all $k$-sets contained in a fixed $(sk-1)$-set.  For
$s=2$, this reduces to the classical Erd\H{o}s--Ko--Rado
theorem~\cite{ErdosKoRado}.  The Hilton--Milner theorem~\cite{HiltonMilner}
determines the largest intersecting family that is not contained in a full
star.  These matching, intersection and nontrivial-intersection problems have
analogues in many finite structures; here we consider the symmetric group.

Let $S_n$ be the symmetric group on $[n]=\{1,\ldots,n\}$.  We identify a
permutation $\sigma\in S_n$ with the perfect matching
$
   M_\sigma=\{(i,\sigma(i)):i\in[n]\}
$
of the complete bipartite graph $K_{n,n}$.  Permutations $\sigma$ and $\tau$
are \emph{disjoint} if $M_\sigma\cap M_\tau=\varnothing$.  Equivalently, the
composition $\sigma\tau^{-1}$ has no fixed point, or
$\sigma(i)\neq\tau(i)$ for every $i\in[n].$
A matching in $\cF\subseteq S_n$ is a subfamily of pairwise disjoint
permutations, and its maximum size is denoted by $\nu(\cF)$.  Thus
$\nu(\cF)<2$ means that $\cF$ is intersecting.

Deza and Frankl~\cite{FranklDeza} proved that an intersecting family in
$S_n$ has size at most $(n-1)!$.  For $1\leq t\leq n$, let
$i_1,\ldots,i_t$ be distinct elements of $[n]$, and let
$j_1,\ldots,j_t$ be distinct elements of $[n]$.  The family
\[
   S_n(i_1\mapsto j_1,\ldots,i_t\mapsto j_t)
      =\{\sigma\in S_n:\sigma(i_r)=j_r\text{ for every }r\in[t]\}
\]
is called a \emph{$t$-coset}; it has size $(n-t)!$.  In particular, a
$1$-coset prescribes the following: for $i,j\in[n]$, we
write
\[
   S_n(i\mapsto j)=\{\sigma\in S_n:\sigma(i)=j\}.
\]
It is an intersecting family of size $(n-1)!$.  Deza and Frankl conjectured
that the $1$-cosets are the only equality cases.  This was proved
independently by Cameron and Ku~\cite{CameronKu}
and by Larose and Malvenuto~\cite{LaroseMalvenuto}.
Ellis subsequently proved the Cameron--Ku stability conjecture and its
Hilton--Milner-type strengthening for all sufficiently large
$n$~\cite{EllisCameronKu}.

More generally, a family is \emph{$t$-intersecting} if any two of its
members agree in at least $t$ positions.  Every $t$-coset is
$t$-intersecting.  Ellis, Friedgut and
Pilpel~\cite{EFP} proved that, for fixed $t$ and sufficiently large $n$, a
$t$-intersecting family has size at most $(n-t)!$.  The equality
classification, also stated in their paper, follows from the stability theorem
of Ellis~\cite{EllisStability}; Filmus~\cite{FilmusComment} later pointed out
the gap in the original proof of this part.  The complete problem is delicate,
just as in the Complete Intersection Theorem of Ahlswede and
Khachatrian~\cite{AhlswedeKhachatrian} for uniform set systems.  For an integer
$r$ with $0\leq r\leq\lfloor(n-t)/2\rfloor$, let
\[
   \cG_{n,t,r}
      =\bigl\{\sigma\in S_n:
          |\{i\in[t+2r]:\sigma(i)=i\}|\geq t+r\bigr\}.
\]
The spread-approximation method introduced by Kupavskii and
Zakharov~\cite{KupavskiiZakharov} first allowed $t$ to grow almost linearly
with $n$, up to polylogarithmic losses.  Keller, Lifshitz, Minzer and
Sheinfeld~\cite{KellerLifshitzMinzerSheinfeld} subsequently proved the
$t$-coset bound, together with a stability statement, throughout the linear
range $n\geq c_0t$ for an absolute constant $c_0$.
Kupavskii~\cite{KupavskiiAlmost} determined the largest $t$-intersecting
families when $n>(1+\varepsilon)t$ and $n$ is sufficiently large in terms
of $\varepsilon$.  More recently, Keller, Kupavskii, Lifshitz and
Sheinfeld~\cite{KellerKupavskiiLifshitzSheinfeld} proved that there is an
absolute $n_0$ such that, for every $n>n_0$ and every $1\leq t\leq n$, the
maximum is attained by one of the families $\cG_{n,t,r}$, up to the natural
left and right actions of $S_n$.  Thus the complete $t$-intersection
problem is now settled in $S_n$ for all sufficiently large $n$.

Averaging over all one-factorizations $L\in\cL_n$, each of which consists
of $n$ pairwise disjoint permutations, gives
\begin{equation}
   \nu(\cF)<s
   \quad\Longrightarrow\quad
   |\cF|\leq(s-1)(n-1)!.
   \label{eq:classical-bound}
\end{equation}
There are two types of equality examples.  If $i\in[n]$ and
$J\subseteq[n]$ has size $s-1$, then
\[
   \{\sigma\in S_n:\sigma(i)\in J\}
      =\bigcup_{j\in J}S_n(i\mapsto j)
\]
has matching number at most $s-1$.  Dually, if $j\in[n]$ and
$I\subseteq[n]$ has size $s-1$, then the same is true of
\[
   \{\sigma\in S_n:\sigma^{-1}(j)\in I\}
      =\bigcup_{i\in I}S_n(i\mapsto j).
\]
Using the method of spread approximations~\cite{KupavskiiZakharov},
Inozemtsev, Kolupaev and
Kupavskii~\cite{InozemtsevKolupaevKupavskii} proved that these are
the only equality cases in the range
$
   s\leq n/(2^{17}\log n);
$
their approach also yields stability and Hilton--Milner-type results.  Our
first theorem settles the equality problem throughout the natural range
$n\geq s\geq2$.

\begin{theorem}\label{thm:extremal}
Let $n\geq s\geq2$, and let $\cF\subseteq S_n$ satisfy
$\nu(\cF)<s$.  Then
\[
   |\cF|\leq(s-1)(n-1)!.
\]
Equality holds if and only if one of the following holds:
\begin{enumerate}[label=\textup{(\roman*)}]
\item there are $i\in[n]$ and $J\subseteq[n]$, with $|J|=s-1$, such that
$
   \cF=\{\sigma\in S_n:\sigma(i)\in J\};
$
\item there are $j\in[n]$ and $I\subseteq[n]$, with $|I|=s-1$, such that
$
   \cF=\{\sigma\in S_n:\sigma^{-1}(j)\in I\}.
$
\end{enumerate}

\end{theorem}

We next describe the function-space statement behind
Theorem~\ref{thm:extremal}.  A \emph{one-factorization} of $K_{n,n}$ is a
set $L\subseteq S_n$ of $n$ pairwise disjoint permutations; write $\cL_n$
for the set of all such $L$.  A \emph{Latin square of order $n$} is an
$n\times n$ array with entries from a set of $n$ symbols such that every
symbol occurs exactly once in each row and exactly once in each column.
After labelling the factors of a one-factorization by these symbols, place in
cell $(i,j)$ the label of the unique factor containing the edge $(i,j)$.
This produces a Latin square of order $n$; conversely, the symbol classes of
a Latin square form a one-factorization.  Thus one-factorizations provide the
natural linear encoding of Latin squares used below.  For
$i,j\in[n]$, define the point-coordinate function
\[
   x_{i,j}(\sigma)=\mathbf1_{\{\sigma(i)=j\}},
\]
and let
\begin{equation}
   U_n=\Span_{\R}\{x_{i,j}:i,j\in[n]\}.
   \label{eq:Un-intro}
\end{equation}
Since $x_{i,j}$ is the indicator of $S_n(i\mapsto j)$, the space $U_n$ is
precisely the real linear span of the indicators of all $1$-cosets.
For every fixed $i$,
$
   \sum_{j=1}^n x_{i,j}=1,
$
so $U_n$ contains the constant functions.

Equip $\R^{S_n}$ with the standard inner product
\[
   \langle f,g\rangle
      =\sum_{\sigma\in S_n}f(\sigma)g(\sigma).
\]
For $L\in\cL_n$, let $v_L=\mathbf1_L$ be its incidence function, and define
\begin{equation}
   T_n
      =\Span_{\R}\{v_L-v_{L'}:L,L'\in\cL_n\}.
   \label{eq:Tn}
\end{equation}
Thus $T_n^\perp$ is precisely the space of functions whose sum is the same
on every one-factorization.  Moreover, every one-factorization contains
exactly one permutation mapping $i$ to $j$.  Consequently,
$U_n\subseteq T_n^\perp$.

Boolean-valued functions in $U_n$ are already completely understood:
they are precisely indicators of unions of $1$-cosets in one row or in
one column~\cite{EFP,FilmusNearlyLinear}.  Thus the equality problem leads to the natural linear statement that a function has constant sum on every one-factorization if and only if it belongs to $U_n$. Our second main theorem establishes this statement for every order. 

\begin{theorem}\label{thm:constant-sum}
For every $n\geq2$,
\[
   U_n=T_n^\perp.
\]
In other words, a function $f:S_n\to\R$ belongs to $U_n$ if and only if
$\sum_{\sigma\in L}f(\sigma)$ is independent of the one-factorization $L$.
Equivalently, this holds if and only if there are $c\in\R$ and
$a_{i,j}\in\R$ such that
\[
   f(\sigma)=c+\sum_{i=1}^n a_{i,\sigma(i)}
   \qquad(\sigma\in S_n).
\]
\end{theorem}

\begin{corollary}\label{cor:factorization-span-dimensions}
For every $n\geq2$, the incidence functions of the one-factorizations satisfy
\begin{align*}
   \dim\Span_{\R}\{\mathbf1_L:L\in\cL_n\}
      &=n!-(n-1)^2,\\
   \dim\Span_{\R}\{\mathbf1_L-\mathbf1_{L'}:L,L'\in\cL_n\}
      &=n!-(n-1)^2-1.
\end{align*}
\end{corollary}

\begin{remark}
There is an equivalent rational group-algebra formulation.  Let $P_\sigma$
be the permutation matrix of $\sigma$, and define
\[
   \Phi_{n,\mathbb Q}:\mathbb Q[S_n]\longrightarrow M_n(\mathbb Q),
   \qquad
   \Phi_{n,\mathbb Q}\left(\sum_{\sigma\in S_n}c_\sigma\sigma\right)
      =\sum_{\sigma\in S_n}c_\sigma P_\sigma.
\]
Since $\Phi_{n,\mathbb Q}(\sum_{\sigma\in L}\sigma)=J$ for every
$L\in\cL_n$, where
$J$ is the all-ones matrix, one may ask whether
\begin{equation}
   \Span_{\mathbb Q}
      \left\{\sum_{\sigma\in L}\sigma:L\in\cL_n\right\}
      =\Phi_{n,\mathbb Q}^{-1}(\mathbb QJ).
   \label{eq:latin-square-span-intro}
\end{equation}
Cooper and Dukes~\cite{CooperDukes} recently proved
\eqref{eq:latin-square-span-intro} for all sufficiently large $n$ and
verified it computationally for $n\leq11$; equivalently, the Latin-square
vectors span a space of dimension $n!-(n-1)^2$.  This is precisely the
large-$n$ case of the linear statement above.  Indeed, after
fixing $L_0\in\cL_n$, the left-hand side of
\eqref{eq:latin-square-span-intro} is
\[
   \mathbb Q\sum_{\sigma\in L_0}\sigma
      +\Span_{\mathbb Q}
         \left\{\sum_{\sigma\in L}\sigma
                    -\sum_{\sigma\in L_0}\sigma:L\in\cL_n\right\},
\]
whereas the right-hand side is
$\mathbb Q\sum_{\sigma\in L_0}\sigma+\ker\Phi_{n,\mathbb Q}$.  Thus
\eqref{eq:latin-square-span-intro} is equivalent to saying that the
differences of one-factorizations span $\ker\Phi_{n,\mathbb Q}$.  Extending
scalars from $\mathbb Q$ to $\mathbb R$ and taking orthogonal complements
gives Theorem~\ref{thm:constant-sum}.  Conversely, both spaces are defined
over $\mathbb Q$, so the real identity proved in this paper implies the
rational identity above.  Hence our theorem establishes
\eqref{eq:latin-square-span-intro} for every $n\geq2$.
\end{remark}

Here is the connection between the two theorems.  Equality in
\eqref{eq:classical-bound} forces
$
   |\cF\cap L|=s-1
$
for every one-factorization $L$.  Theorem~\ref{thm:constant-sum} therefore
places $\mathbf1_{\cF}$ in $U_n$, and the Boolean classification then gives
the two families in Theorem~\ref{thm:extremal}.

In contrast to the representation-theoretic approach above, our proof of
Theorem~\ref{thm:constant-sum} is combinatorial and uses only elementary
linear algebra and basic facts about permutations.  Its advantage is that it
applies to every $n$ and exposes a local-to-global mechanism behind
$U_n$.  In the notation above, the constant-sum condition is precisely
orthogonality to $T_n$, so it is enough to prove
$T_n^\perp\subseteq U_n$.  We first complete two pairs of disjoint
permutations that cover the same edge set to two one-factorizations.  The
resulting four-term relations imply that, for $f\in T_n^\perp$, the difference
$
   f(\sigma)-f(\sigma(a\ b))
$
depends only on the positions $a,b$ and the two values $\sigma(a),\sigma(b)$.
We say that $f$ has local transposition differences.  We then compare a direct
transposition with a three-step realization of the same transposition.  The
resulting cocycle identity shows that every
transposition difference is a difference of one-variable potentials.  Since
adjacent transpositions generate $S_n$, these potentials give the additive
representation in Theorem~\ref{thm:constant-sum}.

The paper is organized as follows.  Section~2 gives the
one-factorization averaging argument, recalls $U_n$ and $T_n$, and
reduces both main theorems to the inclusion $T_n^\perp\subseteq U_n$.
Section~3 proves this inclusion in two steps: first, functions in
$T_n^\perp$ have local transposition differences; second, every function
with local transposition differences belongs to $U_n$.  Section~4 uses these
conclusions to prove Theorems~\ref{thm:constant-sum} and~\ref{thm:extremal},
and Corollary~\ref{cor:factorization-span-dimensions}.
Section~5 discusses the general Hilton--Milner problem for matchings in
permutations.

\section{The reduction to a function-space problem}

\subsection{One-factorizations and equality}

Recall that $\cL_n$ is the set of all one-factorizations of $K_{n,n}$,
represented as sets of permutations.  This set is nonempty: if $\gamma$ is an
$n$-cycle, then
\begin{equation}
   L_\gamma=\{e,\gamma,\ldots,\gamma^{n-1}\}
   \label{eq:cyclic-factorization}
\end{equation}
is a one-factorization.  We compose permutations from right to left.  For
$g,h\in S_n$ and $L\in\cL_n$, let
\[
   gLh=\{g\sigma h:\sigma\in L\}.
\]
For each position $i$, the values $g\sigma h(i)$, as $\sigma$ ranges over
$L$, are all distinct.  Hence $gLh\in\cL_n$.  In particular, left and right
composition preserve $\cL_n$, and left composition acts transitively on
$S_n$.

\begin{lemma}\label{lem:average}
For every function $f:S_n\to\R$,
\begin{equation}
   \frac1{|\cL_n|}\sum_{L\in\cL_n}\sum_{\sigma\in L}f(\sigma)
      =\frac1{(n-1)!}\sum_{\sigma\in S_n}f(\sigma).
   \label{eq:average}
\end{equation}
\end{lemma}

\begin{proof}
Let $r_n$ be the number of one-factorizations containing a fixed
permutation.  The number is independent of the permutation because left
composition preserves $\cL_n$ and acts transitively on $S_n$.  Counting
pairs $(\sigma,L)$ with $\sigma\in L$ gives
$
   n|\cL_n|=n!r_n.
$
Therefore
\[
   \sum_{L\in\cL_n}\sum_{\sigma\in L}f(\sigma)
      =r_n\sum_{\sigma\in S_n}f(\sigma)
      =\frac{|\cL_n|}{(n-1)!}\sum_{\sigma\in S_n}f(\sigma),
\]
which is \eqref{eq:average}.
\end{proof}

\begin{proposition}\label{prop:equality-reduction}
Let $n\geq2$, $s\geq2$, and $\cF\subseteq S_n$ satisfy $\nu(\cF)<s$.
Then
\[
   |\cF|\leq(s-1)(n-1)!.
\]
If equality holds, then $s\leq n+1$ and
\begin{equation}
   |\cF\cap L|=s-1
   \qquad\text{for every }L\in\cL_n.
   \label{eq:saturation}
\end{equation}
\end{proposition}

\begin{proof}
Every $L\in\cL_n$ is a matching of size $n$, so
\[
   |\cF\cap L|\leq\min\{s-1,n\}\leq s-1.
\]
Applying Lemma~\ref{lem:average} to $\mathbf1_{\cF}$ gives
\[
   \frac{|\cF|}{(n-1)!}
      =\frac1{|\cL_n|}\sum_{L\in\cL_n}|\cF\cap L|
      \leq s-1.
\]
If equality holds, every term in the average equals $s-1$, proving
\eqref{eq:saturation}.  Since $|\cF\cap L|\leq n$, equality also gives
$s-1\leq n$.
\end{proof}

\subsection{The spaces \texorpdfstring{$U_n$}{U-n} and
\texorpdfstring{$T_n$}{T-n}}

Recall the spaces $U_n$ and $T_n$ from
\eqref{eq:Un-intro} and~\eqref{eq:Tn}, respectively.
For $\sigma\in S_n$, let $\delta_\sigma$ be the function that is $1$ at
$\sigma$ and $0$ elsewhere.  Thus the functions $\delta_\sigma$ form an
orthonormal basis of $\R^{S_n}$.

We shall use the following known classification of the Boolean-valued members
of $U_n$.  It is the
degree-one case proved by Ellis, Friedgut and Pilpel
\cite[Theorem~28 and Corollary~2]{EFP}; see also
Filmus~\cite[Theorem~2.8]{FilmusNearlyLinear}.

\begin{theorem}[Ellis--Friedgut--Pilpel~\cite{EFP}]\label{thm:boolean-linear}
Let $f:S_n\to\{0,1\}$ belong to $U_n$.  Then one of the following holds:
\begin{enumerate}[label=\textup{(\roman*)}]
\item there are $i\in[n]$ and $J\subseteq[n]$ such that
$
   f(\sigma)=\mathbf1_{\{\sigma(i)\in J\}};
$
\item there are $j\in[n]$ and $I\subseteq[n]$ such that
$
   f(\sigma)=\mathbf1_{\{\sigma^{-1}(j)\in I\}}.
$
\end{enumerate}
\end{theorem}

For $L\in\cL_n$, we have
$v_L=\sum_{\sigma\in L}\delta_\sigma$.
By definition, a function $f:S_n\to\R$ has the same sum on every
one-factorization if and only if
\[
   \langle f,v_L-v_{L'}\rangle=0
   \qquad(L,L'\in\cL_n),
\]
or equivalently if and only if
\begin{equation}
   f\in T_n^\perp.
   \label{eq:constant-sum-Tperp}
\end{equation}

\begin{lemma}\label{lem:easy-inclusion}
For every $n\geq2$,
\begin{equation}
   U_n\subseteq T_n^\perp.
   \label{eq:easy-orthogonal-inclusion}
\end{equation}
\end{lemma}

\begin{proof}
Fix $i,j\in[n]$ and $L\in\cL_n$.  Since the permutations in $L$ are
pairwise disjoint, the $n$ values $\sigma(i)$, with $\sigma\in L$, are
pairwise distinct and hence comprise all of $[n]$.  Thus exactly one
$\sigma\in L$ satisfies $\sigma(i)=j$, and therefore
\[
   \langle x_{i,j},v_L\rangle
      =\sum_{\sigma\in L}x_{i,j}(\sigma)=1.
\]
Consequently, for all $L,L'\in\cL_n$,
$
   \langle x_{i,j},v_L-v_{L'}\rangle=0.
$
Hence every $x_{i,j}$ is orthogonal to $T_n$.  Since these functions span
$U_n$, we obtain \eqref{eq:easy-orthogonal-inclusion}.
\end{proof}

Lemma~\ref{lem:easy-inclusion} and
\eqref{eq:constant-sum-Tperp} reduce
Theorem~\ref{thm:constant-sum} to the single reverse inclusion
\begin{equation}
   T_n^\perp\subseteq U_n.
   \label{eq:main-reduction}
\end{equation}
They also reduce the equality case in Theorem~\ref{thm:extremal} to the
same statement: by \eqref{eq:saturation}, equality implies
$\mathbf1_{\cF}\in T_n^\perp$; once \eqref{eq:main-reduction} is proved,
Theorem~\ref{thm:boolean-linear} determines $\cF$.

\section{Local transposition differences}

For distinct positions $a,b\in[n]$, define
\begin{equation}
   \Delta_{ab}f(\sigma)
      =f(\sigma)-f(\sigma(a\ b)).
   \label{eq:transposition-difference}
\end{equation}
The quantity $\Delta_{ab}f(\sigma)$ is the \emph{transposition difference}
of $f$ at $\sigma$ in positions $a,b$.  We say that $f:S_n\to\R$ has
\emph{local transposition differences} if,
for every $a\neq b$, there is a function
\[
   D_{ab}:\{(x,y)\in[n]^2:x\neq y\}\longrightarrow\R
\]
such that
\begin{equation}
   \Delta_{ab}f(\sigma)
      =D_{ab}(\sigma(a),\sigma(b))
   \qquad(\sigma\in S_n).
   \label{eq:transposition-property}
\end{equation}
Thus the change in $f$ caused by exchanging positions $a$ and $b$ depends
only on those positions and the two values $\sigma(a),\sigma(b)$.

\subsection{Functions in \texorpdfstring{$T_n^\perp$}{T-n-perp} have local
transposition differences}

We first record the invariance of $T_n$ under composition.

\begin{lemma}\label{lem:translation}
Suppose
\[
   w=\sum_{\sigma\in S_n}c_\sigma\delta_\sigma\in T_n.
\]
Then, for every $g,h\in S_n$,
\[
   \sum_{\sigma\in S_n}c_\sigma\delta_{g\sigma h}\in T_n.
\]
\end{lemma}

\begin{proof}
It is enough to check a spanning vector $w=v_L-v_{L'}$ of $T_n$.  In that
case the displayed function is
$
   v_{gLh}-v_{gL'h},
$
which belongs to $T_n$ because $gLh,gL'h\in\cL_n$.
\end{proof}

The next relation is the local input for the proof.

\begin{lemma}\label{lem:four-term}
Let $n\geq4$.  Fix $A=\{a,b\}\subseteq[n]$, let $\alpha=(a\ b)$, and let
$\beta\in S_n$ fix $a$ and $b$ and have no fixed point in
$[n]\setminus A$.  Then
\begin{equation}
   \delta_e+\delta_{\alpha\beta}
      -\delta_\alpha-\delta_\beta\in T_n.
   \label{eq:four-term}
\end{equation}
\end{lemma}

\begin{proof}
The permutations $\alpha$ and $\beta$ have disjoint supports and hence
commute.  The permutation $\alpha\beta$ has no fixed point: $\alpha$ moves
$a,b$, while $\beta$ moves every point outside $A$.  Hence $e$ and
$\alpha\beta$ are disjoint.  Also, if $i\in A$, then
$\alpha(i)\neq i=\beta(i)$, whereas if $i\notin A$, then
$\alpha(i)=i\neq\beta(i)$.  Thus $\alpha$ and $\beta$ are disjoint.
Moreover, for every $i\in[n]$,
$
   \{i,\alpha\beta(i)\}=\{\alpha(i),\beta(i)\}.
$
Thus the two pairs cover the same set of $2n$ edges of $K_{n,n}$.

Delete these edges from $K_{n,n}$.  The remaining graph is
$(n-2)$-regular and bipartite.  Repeated application of Hall's theorem
decomposes it into perfect matchings; let $Q$ be such a decomposition.
Then
\[
   L_+=Q\cup\{e,\alpha\beta\},
   \qquad
   L_-=Q\cup\{\alpha,\beta\}
\]
are one-factorizations.  Therefore
$
   v_{L_+}-v_{L_-}
      =\delta_e+\delta_{\alpha\beta}
         -\delta_\alpha-\delta_\beta\in T_n.
$
\end{proof}

A permutation with no fixed point is a \emph{derangement}.

\begin{lemma}\label{lem:transposition-nonexceptional}
Let $f\in T_n^\perp$.  If $n=4$ or $n\geq6$, then $f$ has local
transposition differences.
\end{lemma}

\begin{proof}
Fix distinct $a,b\in[n]$, let $A=\{a,b\}$ and
$B=[n]\setminus A$, and set $\alpha=(a\ b)$.  Let $\beta$ fix $A$ and be
a derangement on $B$.  Lemmas~\ref{lem:translation}
and~\ref{lem:four-term} give
\[
   \delta_g+\delta_{g\alpha\beta}
      -\delta_{g\alpha}-\delta_{g\beta}\in T_n
   \qquad(g\in S_n).
\]
Taking the inner product with $f$ and using $\alpha\beta=\beta\alpha$
gives
\begin{equation}
   \Delta_{ab}f(g)=\Delta_{ab}f(g\beta).
   \label{eq:invariance-derangement}
\end{equation}

Let $S_B$ be the subgroup of permutations fixing $[n]\setminus B$
pointwise.  If $n\geq6$, then $|B|\geq4$.  It is well known that the
derangements of $S_B$ generate $S_B$, and hence
\eqref{eq:invariance-derangement} implies
\begin{equation}
   \Delta_{ab}f(g)=\Delta_{ab}f(gh)
   \qquad(g\in S_n,\ h\in S_B).
   \label{eq:invariance-complement}
\end{equation}
If $n=4$, then $|B|=2$ and its unique nonidentity permutation is a
derangement, so \eqref{eq:invariance-complement} also holds.

For $g,g'\in S_n$, there is an $h\in S_B$ with $g'=gh$ if and only if
   $g(a)=g'(a)$ and $g(b)=g'(b).$
Indeed, the forward implication is immediate, while the reverse implication
follows because $g^{-1}g'$ fixes $a$ and $b$.  Hence
\eqref{eq:invariance-complement} shows that
$\Delta_{ab}f(g)$ depends only on $g(a)$ and $g(b)$.  Since $a,b$ were
arbitrary, $f$ has local transposition differences.
\end{proof}

For $n=5$, the derangements on the three complementary positions generate
only the alternating group on those positions.  The following relation
provides the missing transpositions.

\begin{lemma}\label{lem:order-five-relation}
If $\alpha$ and $\tau$ are disjoint transpositions in $S_5$, then
\begin{equation}
   \delta_e+\delta_{\alpha\tau}
      -\delta_\alpha-\delta_\tau\in T_5.
   \label{eq:order-five-relation}
\end{equation}
\end{lemma}

\begin{proof}
It is enough to take $\alpha=(1\ 2)$ and $\tau=(3\ 4)$.  We use the following
notation: a word $a_1a_2a_3a_4a_5$ represents the permutation
$\sigma\in S_5$ satisfying
\[
   (\sigma(1),\sigma(2),\sigma(3),\sigma(4),\sigma(5))
      =(a_1,a_2,a_3,a_4,a_5).
\]
Thus, for example, $23154$ represents the permutation with
$(\sigma(1),\ldots,\sigma(5))=(2,3,1,5,4)$.  Each $L_r$ below is a
one-factorization written in this notation:
\begin{align*}
L_1&=\{12345,23154,35412,41523,54231\},\\
L_2&=\{13524,21435,35241,42153,54312\},\\
L_3&=\{12435,23154,35241,41523,54312\},\\
L_4&=\{13524,21345,35412,42153,54231\}.
\end{align*}
Thus,
\begin{align*}
   v_{L_1}+v_{L_2}-v_{L_3}-v_{L_4}
      &=\delta_{12345}+\delta_{21435}
         -\delta_{12435}-\delta_{21345}\\
      &=\delta_e+\delta_{(1\ 2)(3\ 4)}
         -\delta_{(3\ 4)}-\delta_{(1\ 2)}.
\end{align*}
The left-hand side is
$(v_{L_1}-v_{L_3})+(v_{L_2}-v_{L_4})\in T_5$.

For arbitrary disjoint transpositions $\alpha,\tau$, choose
$\pi\in S_5$ such that
\[
   \alpha=\pi(1\ 2)\pi^{-1},
   \qquad
   \tau=\pi(3\ 4)\pi^{-1}.
\]
Apply Lemma~\ref{lem:translation} with $g=\pi$ and $h=\pi^{-1}$ to obtain
\eqref{eq:order-five-relation}.
\end{proof}

\begin{proposition}\label{prop:Tperp-transposition}
Let $n\geq4$.  Every $f\in T_n^\perp$ has local transposition differences.
\end{proposition}

\begin{proof}
Lemma~\ref{lem:transposition-nonexceptional} proves the assertion for $n=4$ and
for $n\geq6$.  It remains to consider $n=5$.

Fix distinct $a,b\in[5]$, let $B=[5]\setminus\{a,b\}$, and set
$\alpha=(a\ b)$.  For every transposition $\tau$ supported on $B$,
Lemma~\ref{lem:order-five-relation} and its left translate by $g\in S_5$
give
\[
   \delta_g+\delta_{g\alpha\tau}
      -\delta_{g\alpha}-\delta_{g\tau}\in T_5.
\]
Taking the inner product with $f\in T_5^\perp$ and using
$\alpha\tau=\tau\alpha$ yields
$
   \Delta_{ab}f(g)=\Delta_{ab}f(g\tau).
$
The transpositions supported on $B$ generate $S_B$.  Hence
$\Delta_{ab}f$ is invariant under right composition by every member of
$S_B$.  As in the final paragraph of the proof of
Lemma~\ref{lem:transposition-nonexceptional}, this means that
$\Delta_{ab}f(g)$ depends only on $g(a)$ and $g(b)$.  Since $a,b$ were
arbitrary, $f$ has local transposition differences.
\end{proof}

\subsection{Functions with local transposition differences are in
\texorpdfstring{$U_n$}{U-n}}

We first isolate a finite cocycle calculation.

\begin{lemma}\label{lem:difference}
Let $X$ be a set with $|X|\geq4$, and let
\[
   d:\{(x,y)\in X^2:x\neq y\}\longrightarrow\R
\]
satisfy $d(y,x)=-d(x,y)$.  Suppose that, for every $x\neq z$, the value
$
   d(x,y)+d(y,z)
$
is independent of $y\in X\setminus\{x,z\}$.  Then there is a function
$u:X\to\R$ such that
\[
   d(x,y)=u(x)-u(y)
   \qquad(x\neq y).
\]
\end{lemma}

\begin{proof}
For $x\neq z$, define
$
   H(x,z)=d(x,y)+d(y,z),
$
where $y\notin\{x,z\}$.  The hypothesis makes $H$ well defined, and
$H(z,x)=-H(x,z)$.  If $x,y,z$ are distinct, choose
$w\notin\{x,y,z\}$.  Then
\begin{align*}
   H(x,y)+H(y,z)
      =d(x,w)+d(w,y)+d(y,w)+d(w,z)
      =H(x,z).
\end{align*}
Fix $x_0\in X$, set $u(x_0)=0$, and set
$u(x)=H(x,x_0)$ for $x\neq x_0$.  If $x,z\neq x_0$, then the cocycle
identity gives
$H(x,z)+H(z,x_0)=H(x,x_0)$.  If $z=x_0$, the desired identity follows
from the definition of $u$, and if $x=x_0$, it follows from the
antisymmetry of $H$.  Consequently,
\[
   H(x,z)=u(x)-u(z)
   \qquad(x\neq z).
\]

Let $e(x,y)=d(x,y)-H(x,y)$.  For distinct $x,y,z$,
$
   e(x,y)+e(y,z)=0.
$
Fix $y$, and let $x,x'\in X\setminus\{y\}$.  Choose
$z\notin\{x,x',y\}$, which is possible because $|X|\geq4$.  Applying the
last identity to $(x,y,z)$ and $(x',y,z)$ gives
$e(x,y)=e(x',y)$.  Thus $e(x,y)$ is independent of $x\neq y$; write this
common value as $c_y$.  Antisymmetry gives
$c_y=-c_x$ whenever $x\neq y$.  Taking three distinct elements gives
$c_x=-c_y=c_z=-c_x$, and hence every $c_x$ is zero.  Thus
$e=0$ and $d(x,y)=u(x)-u(y)$.
\end{proof}

\begin{proposition}\label{prop:transposition-additive}
Let $n\geq4$, and let $f:S_n\to\R$ have local transposition differences.
Then there are $c\in\R$ and coefficients $a_{i,x}\in\R$ such that
\begin{equation}
   f(\sigma)=c+\sum_{i=1}^n a_{i,\sigma(i)}
   \qquad(\sigma\in S_n).
   \label{eq:additive-form}
\end{equation}
In particular, $f\in U_n$.
\end{proposition}

\begin{proof}
For $1\leq i\leq n-1$, let $s_i=(i\ i+1)$ and define
\[
   d_i(x,y)=D_{i,i+1}(x,y)
   \qquad(x\neq y).
\]
If $\sigma(i)=x$ and $\sigma(i+1)=y$, then
$
   d_i(x,y)=f(\sigma)-f(\sigma s_i).
$
Applying the same identity to $\sigma s_i$ gives
\begin{equation}
   d_i(y,x)=-d_i(x,y).
   \label{eq:di-antisymmetric}
\end{equation}

Fix $i$ and choose a position $q\notin\{i,i+1\}$.  Let $x,y,z$ be
distinct values, and choose $\sigma\in S_n$ satisfying
\[
   \sigma(i)=x,\qquad
   \sigma(i+1)=y,\qquad
   \sigma(q)=z.
\]
Let $t=(i+1\ q)$.  Since $(i\ q)=s_i t s_i$, telescoping along
\[
   \sigma,\quad
   \sigma s_i,\quad
   \sigma s_i t,\quad
   \sigma s_i t s_i=\sigma(i\ q)
\]
gives
$
   D_{i,q}(x,z)
      =d_i(x,y)+D_{i+1,q}(x,z)+d_i(y,z).
$
Therefore
\begin{equation}
   d_i(x,y)+d_i(y,z)
      =D_{i,q}(x,z)-D_{i+1,q}(x,z),
   \label{eq:two-step-independent}
\end{equation}
which is independent of $y$.  Lemma~\ref{lem:difference} gives a function
$u_i:[n]\to\R$ such that
\begin{equation}
   d_i(x,y)=u_i(x)-u_i(y).
   \label{eq:di-potential}
\end{equation}

Set $a_{1,x}=0$ for every $x$, and recursively define
\[
   a_{i+1,x}=a_{i,x}-u_i(x)
   \qquad(1\leq i\leq n-1).
\]
Let
$
   F(\sigma)=\sum_{k=1}^n a_{k,\sigma(k)}.
$
If $\sigma(i)=x$ and $\sigma(i+1)=y$, then
\begin{align*}
   F(\sigma)-F(\sigma s_i)
      &=a_{i,x}+a_{i+1,y}-a_{i,y}-a_{i+1,x}\\
      &=u_i(x)-u_i(y)\\
      &=f(\sigma)-f(\sigma s_i).
\end{align*}
Thus $f-F$ is invariant under right composition by every adjacent
transposition.  Since the adjacent transpositions generate $S_n$, the
function $f-F$ is constant.  This proves \eqref{eq:additive-form}, whose
right-hand side belongs to $U_n$.
\end{proof}

Combining Propositions~\ref{prop:Tperp-transposition}
and~\ref{prop:transposition-additive}, we have proved the reduction
\eqref{eq:main-reduction} for every $n\geq4$.

\section{Proofs of the main theorems}

\subsection{Proof of Theorem~\ref{thm:constant-sum}}

\begin{proof}
Suppose first that $n\geq4$.  Lemma~\ref{lem:easy-inclusion} gives
$U_n\subseteq T_n^\perp$.  Conversely, if $f\in T_n^\perp$, then
Proposition~\ref{prop:Tperp-transposition} gives local transposition differences,
and Proposition~\ref{prop:transposition-additive} gives $f\in U_n$.
Hence
\begin{equation}
   T_n^\perp=U_n
   \qquad(n\geq4).
   \label{eq:Tperp-Un-large}
\end{equation}

We verify the two smaller orders directly.  For $n=2$, there is a unique
one-factorization, so $T_2=0$.  Each of the two singleton subsets of $S_2$
is a $1$-coset, and therefore $U_2=\R^{S_2}=T_2^\perp$.

Let $n=3$.  Two distinct permutations $\sigma,\tau\in S_3$ are disjoint
if and only if $\sigma^{-1}\tau$ is a derangement.  The two derangements
in $S_3$ are the two $3$-cycles, so distinct permutations are disjoint if
and only if they have the same parity.  It follows that
\[
   L_+=A_3,
   \qquad
   L_-=S_3\setminus A_3
\]
are the two one-factorizations and
$
   T_3=\Span_{\R}\{v_{L_+}-v_{L_-}\}.
$

The space $T_3$ is one-dimensional, so $T_3^\perp$ has dimension $5$.
On the other hand, the five functions
\[
   1,\quad x_{1,1},\quad x_{1,2},\quad x_{2,1},\quad x_{2,2}
\]
belong to $U_3$ and are linearly independent.  Indeed, suppose that
\[
   c+a x_{1,1}+b x_{1,2}+d x_{2,1}+e x_{2,2}=0.
\]
Evaluating this identity at $132$, $231$, $312$, $213$, and $123$, in
that order, gives
\[
   c+a=0,\qquad c+b=0,\qquad c+d=0,
   \qquad c+b+d=0,\qquad c+a+e=0.
\]
The first four equations give $c=a=b=d=0$, and the last then gives
$e=0$.  Hence $\dim U_3\geq5$.  Lemma~\ref{lem:easy-inclusion} now gives
$U_3=T_3^\perp$.

We have proved $T_n^\perp=U_n$ for every $n\geq2$, which is the displayed
identity in the theorem.  By \eqref{eq:constant-sum-Tperp}, it also gives the
constant-sum formulation.  Finally, every function of the displayed additive form
belongs to $U_n$.  Conversely, a linear combination of the generators
$x_{i,j}$ has such a form, while constants already lie in $U_n$.  This
proves the additive formulation as well.

\end{proof}

\subsection{Proof of Theorem~\ref{thm:extremal}}

\begin{proof}
The bound was proved in Proposition~\ref{prop:equality-reduction}.

Assume equality.  By \eqref{eq:saturation}, for every
$L,L'\in\cL_n$,
$
   \langle\mathbf1_{\cF},v_L-v_{L'}\rangle=0.
$
Hence $\mathbf1_{\cF}\in T_n^\perp$.  By
Theorem~\ref{thm:constant-sum}, $\mathbf1_{\cF}\in U_n$.  Since this
function is Boolean, Theorem~\ref{thm:boolean-linear} shows that either
\[
   \cF=\{\sigma\in S_n:\sigma(i)\in J\}
\]
for some $i$ and $J$, or
\[
   \cF=\{\sigma\in S_n:\sigma^{-1}(j)\in I\}
\]
for some $j$ and $I$.  These two families have sizes
$|J|(n-1)!$ and $|I|(n-1)!$, respectively.  Equality therefore forces
$|J|=s-1$ or $|I|=s-1$.

Conversely, let
\[
   \cF=\{\sigma\in S_n:\sigma(i)\in J\},
   \qquad |J|=s-1.
\]
Then $|\cF|=(s-1)(n-1)!$.  If
$\sigma_1,\ldots,\sigma_t$ are pairwise disjoint members of $\cF$, then
$
   \sigma_1(i),\ldots,\sigma_t(i)
$
are distinct elements of $J$, and hence $t\leq s-1$.  Thus
$\nu(\cF)<s$.  The proof for a family of type~\textup{(ii)} is identical,
using the distinct preimages
$\sigma_1^{-1}(j),\ldots,\sigma_t^{-1}(j)$.
\end{proof}

\subsection{Proof of Corollary~\ref{cor:factorization-span-dimensions}}

\begin{proof}
First, $\dim U_n=(n-1)^2+1$.  Indeed, consider the linear map
\[
   \R^{n\times n}\longrightarrow U_n,
   \qquad
   (a_{i,j})\longmapsto
      \left(\sigma\longmapsto\sum_{i=1}^n a_{i,\sigma(i)}\right).
\]
Its kernel consists precisely of the matrices
$a_{i,j}=r_i+c_j$ satisfying
$\sum_i r_i+\sum_j c_j=0$.  Indeed, comparing two permutations that agree
outside positions $i,k$ and transpose the values $j,\ell$ gives the four-cell
relations
$a_{i,j}+a_{k,\ell}=a_{i,\ell}+a_{k,j}$; these relations imply the stated
form, and the converse is immediate.  This kernel has dimension $2n-2$,
proving the stated value of $\dim U_n$.

Let
\[
   W_n^0=\Span_{\R}\{\mathbf1_L-\mathbf1_{L'}:L,L'\in\cL_n\}.
\]
By Theorem~\ref{thm:constant-sum}, $(W_n^0)^\perp=U_n$, and hence
\[
   \dim W_n^0=n!-\dim U_n=n!-(n-1)^2-1.
\]
Every member of $W_n^0$ has total coordinate sum zero, whereas
$\mathbf1_L$ has total coordinate sum $n$.  Moreover, after fixing
$L_0\in\cL_n$, every $\mathbf1_L$ is the sum of $\mathbf1_{L_0}$ and an
element of $W_n^0$.  Thus adjoining $\mathbf1_{L_0}$ contributes exactly
one further dimension, proving the first identity.
\end{proof}

\section{Concluding remarks}

For a family $\cF\subseteq S_n$, define its \emph{covering number} by
\[
   \tau(\cF)
      =\min\bigl\{|X|:X\subseteq[n]\times[n]
          \text{ and }X\cap M_\sigma\neq\varnothing
          \text{ for every }\sigma\in\cF\bigr\}.
\]
A union of $(s-1)$ $1$-cosets has covering number at most $s-1$.
Consequently, the condition $\tau(\cF)\geq s$ is the natural
Hilton--Milner condition for a family with matching number less than $s$.

Let
\[
   \cD_n=\{\pi\in S_n:\pi(i)\neq i\text{ for every }i\in[n]\},
   \qquad
   d_{n,1}=|\{\pi\in\cD_n:\pi(1)=2\}|.
\]
The second quantity is independent of the chosen off-diagonal ordered
pair.  Fix $\eta\in S_n$ with $\eta(1)\notin[s-1]$, and define
\begin{align}
   \cH_{n,s}(\eta)
      =\bigcup_{i=2}^{s-1}S_n(1\mapsto i)\cup
       \{\pi\in S_n(1\mapsto1):M_\pi\cap M_\eta\neq\varnothing\}
       \cup\{\eta\}.
   \label{eq:permutation-HM-family}
\end{align}
This is the union of $(s-2)$ pairwise disjoint $1$-cosets and a permutation
Hilton--Milner family, with $\nu(\cH_{n,s}(\eta))\leq s-1$ and $\tau(\cH_{n,s}(\eta))\geq s$.

The three parts in the definition are pairwise disjoint.  Among the
$(n-1)!$ permutations mapping $1$ to $1$, those disjoint from $\eta$ are in
bijection, via $\pi\mapsto\pi\eta^{-1}$, with the derangements $\delta$
satisfying $\delta(\eta(1))=1$. Therefore
\[
   |\cH_{n,s}(\eta)|
      =(s-1)(n-1)!-d_{n,1}+1.
\]
Inozemtsev, Kolupaev and Kupavskii showed that, in the range below, this
construction has covering number $s$ and is sharp: if
$s\leq n/(2^{17}\log n),$ $ \nu(\cF)<s$ and $\tau(\cF)\geq s,$
then
\begin{equation}
   |\cF|\leq(s-1)(n-1)!-d_{n,1}+1,
   \label{eq:IKK-HM-bound}
\end{equation}
and equality holds only for the images of
$\cH_{n,s}(\eta)$ under left and right composition and inversion
\cite{InozemtsevKolupaevKupavskii}.

The general parameter range remains open.  It is useful to isolate the
problem as follows.

\begin{problem}\label{prob:general-HM}
For $n\geq2$ and $2\leq s\leq n$, determine
\[
   \operatorname{HM}(n,s)
      =\max\bigl\{|\cF|:\cF\subseteq S_n,\
          \nu(\cF)<s,\ \tau(\cF)\geq s\bigr\},
\]
and classify all extremal families.  In particular, determine the full
range of pairs $(n,s)$ for which
\[
   \operatorname{HM}(n,s)
      =(s-1)(n-1)!-d_{n,1}+1
\]
with equality only for the families in
\eqref{eq:permutation-HM-family}, up to the natural symmetries of $S_n$,
and identify the competing constructions when this formula fails.
\end{problem}

\section*{Acknowledgements} The authors are very grateful to Dr. Binzhou Xia for helpful discussions and for valuable suggestions concerning the group-theoretic arguments and their presentation. The authors acknowledge the use of AI tools during the
exploratory stage of this project. All mathematical arguments and proofs
presented in the final manuscript were developed and rigorously verified
by the authors. The authors take full responsibility for the content of
the manuscript.

\end{document}